\documentclass[12pt]{amsart}
\usepackage{amssymb}
\usepackage{bm}
    \usepackage{calrsfs}

     \usepackage{color}

\usepackage{pdfsync}

\newcommand{\calL}{{\mathcal L}}

\newcommand{\pp}{{\mathfrak p}}

\newcommand{\dd}{{\mathfrak d}}
\newcommand{\nn}{{\mathfrak n}}
\newcommand{\DD}{{\mathfrak D}}
\newcommand{\EE}{{\mathfrak E}}
\newcommand{\LL}{{\mathfrak L}}
\newcommand{\NN}{{\mathfrak N}}
\newcommand{\PP}{{\mathfrak P}}

\newcommand{\Z}{\mathbb{Z}}
\newcommand{\Q}{\mathbb{Q}}

\newcommand{\C}{\mathbb{C}}

\newcommand{\ord}{{\rm ord}}
\newcommand{\HH}{{\rm H}}
\newcommand{\LCM}{{\rm lcm}}

\newcommand{\ch}{{\rm char}}

\newtheorem{theorem}{Theorem}[section]
\newtheorem{proposition}[theorem]{Proposition}
\newtheorem{lemma}[theorem]{Lemma}
\newtheorem{corollary}[theorem]{Corollary}
\newtheorem{remark}[theorem]{Remark}

\newtheorem{notation}[theorem]{Notation}
\newtheorem{NotationAssumption}[theorem]{Notation and Assumptions}

\newtheorem{assumption}[theorem]{Assumption}

\begin{document}

\title{The analogue of B\"uchi's problem for function fields}

\author{Alexandra Shlapentokh\\
and\\
Xavier Vidaux}
\date{\today}

\begin{abstract}
B\"uchi's $n$ Squares Problem  asks for an integer $M$ such that any sequence $(x_0,\dots,x_{M-1})$, whose second difference of squares is the constant sequence $(2)$ (i.e. $x^2_n-2x^2_{n-1}+x_{n-2}^2=2$ for all $n$), satisfies $x_n^2=(x+n)^2$ for some integer $x$.  Hensley's problem for $r$-th powers (where $r$ is an integer $\geq2$) is a generalization of B\"{u}chi's problem asking for an integer $M$ such that, given integers $\nu$ and $a$, the quantity $(\nu+n)^r-a$ cannot be an $r$-th power for $M$ or more values of the integer $n$, unless $a=0$. The analogues of these problems for rings of functions consider only sequences with at least one non-constant term.

Let $K$ be a function field of a curve of genus $g$. We prove that Hensley's problem for $r$-th powers has a positive answer for any $r$ if $K$ has characteristic zero, improving results by Pasten and Vojta. In positive characteristic $p$ we obtain a weaker result, but which is enough to prove that B\"uchi's problem has a positive answer if $p\geq 312g+169$ (improving results by Pheidas and the second author).
\end{abstract}

\maketitle
%%%%%%%%%%%%%%%%%%%%%%%%%%%%%%%%%%%%%%%%%
%%%%%%%%%%%%%%%%%%%%%%%%%%%%%%%%%%%%%%%%%
%%%%%%%%%%%%%%%%%%%%%%%%%%%%%%%%%%%%%%%%%
%%%%%%%%%%%%%%%%%%%%%%%%%%%%%%%%%%%%%%%%%
%%%%%%%%%%%%%%%%%%%%%%%%%%%%%%%%%%%%%%%%%
%%%%%%%%%%%%%%%%%%%%%%%%%%%%%%%%%%%%%%%%%
%%%%%%%%%%%%%%%%%%%%%%%%%%%%%%%%%%%%%%%%%
%%%%%%%%%%%%%%%%%%%%%%%%%%%%%%%%%%%%%%%%%
%%%%%%%%%%%%%%%%%%%%%%%%%%%%%%%%%%%%%%%%%
%%%%%%%%%%%%%%%%%%%%%%%%%%%%%%%%%%%%%%%%%
%%%%%%%%%%%%%%%%%%%%%%%%%%%%%%%%%%%%%%%%%
%%%%%%%%%%%%%%%%%%%%%%%%%%%%%%%%%%%%%%%%%
%%%%%%%%%%%%%%%%%%%%%%%%%%%%%%%%%%%%%%%%%
%%%%%%%%%%%%%%%%%%%%%%%%%%%%%%%%%%%%%%%%%
%%%%%%%%%%%%%%%%%%%%%%%%%%%%%%%%%%%%%%%%%
%%%%%%%%%%%%%%%%%%%%%%%%%%%%%%%%%%%%%%%%%
%%%%%%%%%%%%%%%%%%%%%%%%%%%%%%%%%%%%%%%%%
%%%%%%%%%%%%%%%%%%%%%%%%%%%%%%%%%%%%%%%%%
%%%%%%%%%%%%%%%%%%%%%%%%%%%%%%%%%%%%%%%%%
%%%%%%%%%%%%%%%%%%%%%%%%%%%%%%%%%%%%%%%%%
%%%%%%%%%%%%%%%%%%%%%%%%%%%%%%%%%%%%%%%%%
%%%%%%%%%%%%%%%%%%%%%%%%%%%%%%%%%%%%%%%%%

AMS Subject Classification\,: 03B25, 11D41, 11U05

\section{Introduction}

A  1990 paper by L. Lipshitz \cite{Lipshitz} containing a description of a question posed in the 70's by J. R. B\"uchi  inspired a new interest in what is known today as ``B\"uchi's Problem'' or ``the $n$ Squares Problem''\, (denoted by $\mathbf{B^2}(\Z)$ in the future)\,:\\

 \noindent\emph{Does there exist a positive integer $M$ such that any sequence of $M$ integer squares, whose second difference is constant and equal to $2$, is of the form $(x+n)^2$, $n=1,\dots,M$, for some integer $x$?}\\

B\"uchi asked this question because a positive answer to it would imply a stronger form of the negative answer to Hilbert's Tenth Problem solved in 1970 by Yuri Matiyasevich using results of Martin Davis, Hilary Putnam and Julia Robinson. In logical terms, Matiyasevich's result (see \cite{Matiyasevic} and \cite{Davis}) implies that the positive existential theory of $\Z$ in the language $\calL=\{0,1,+,\cdot\}$ of rings is undecidable. B\"uchi observed that a positive answer to his problem would allow him to define existentially the multiplication over $\Z$ in the language $\calL^2=\{0,1,+,P_2\}$, where $P_2$ is a unary predicate for ``$x$ is a square'', hence proving that the positive existential theory of $\Z$ in the language $\calL^2$ is undecidable.

It makes sense to ask B\"uchi's question over other rings. If $R$ is a commutative ring with identity, the problem $\mathbf{B^2}(R)$ becomes\,:\\

 \noindent\emph{Does there exist a positive integer $M$ such that any sequence of $M$ squares in $R$, whose second difference is constant and equal to $2$, is of the form $(x+n)^2$, $n=0,\dots,M-1$, for some $x\in R$?}\\

A positive (or ``almost positive'') answer to $\mathbf{B^2}(R)$ \emph{in general} has similar logical consequences to a positive answer to $\mathbf{B^2}(\Z)$ if the existential ring theory of $R$ is undecidable.

B\"uchi's Problem is still open. However, in 2001, Vojta proved in \cite{Vojta2} that $\mathbf{B^2}(\Q)$, and hence also $\mathbf{B^2}(\Z)$, have a positive answer for some $M\geq8$ if the following (open) question of Bombieri has a positive answer for surfaces\,:\\

\noindent\emph{Let $X$ be a smooth projective algebraic variety of general type, defined over a number field $k$. Does there exist a proper Zariski-closed subset $Z$ of $X$ such that $X(k)\subseteq Z$?}\\

Vojta's proof actually is valid for any number field as was first noted by Yamagishi in \cite{Yamagishi}. Continuing this line  of investigation, in 2009, Pasten in \cite{Pasten2} produced the following generalization of Vojta's result\,:\\

\noindent\emph{If Bombieri's Question has a positive answer, then there exists an absolute constant $N$ (that can be chosen to be $8$ if Bombieri's question is true for any surface) such that, for each number field $K/\Q$ and each set $\{a_1,\ldots,a_N\}$ of $N$ elements in $K$, there is only a finite number of polynomials $f=x^2+ax+b\in K[x]$ not of the form $f=(x+c)^2$, satisfiying that $f(a_i)$ are squares in $K$ for each $i$.}\\

At the same time, R. G. E. Pinch in \cite{Pinch} proved that `many' non-trivial B\"uchi sequences of length $4$ could not be extended to B\"uchi sequences of length $5$ (originally B\"uchi asked his question for $M=5$).

Before turning our attention to rings of functions, we should note that a number of people (Allison \cite{Allison} in 1986, Bremner \cite{Bremner} in 2003, and Browkin and Brzezinski \cite{BrowkinBrzezinski} in 2006) have been studying the following analogue of B\"uchi's problem\,:\\

\noindent \emph{Does there exist an integer $M$ such that the system of equations
$$
x_{n+2}^2-2x_{n+1}^2+x_n^2=\ell,\quad n=0,\dots,M-3,
$$
where $\ell\in\Z$, has only solutions whose squares are the squares of an arithmetic progression?}\\

Observe that this problem is related to the original B\"uchi's problem over an integral extension of $\Z$\,: multiply the equations by $2\ell^{-1}$ and consider the change of variables
$$
y_n=\frac{\sqrt2}{\sqrt\ell}x_n.
$$

In \cite{Vojta2} Vojta also considered analogues of B\"uchi's Problem over rings of functions. If $R_t$ is a ring of functions in the variable $t$, the problem $\mathbf{B^2}(R_t)$ becomes\,:\\

\noindent\emph{Does there exist a positive integer $M$ such that any sequence of $M$ squares in $R_t$, not all constant, whose second difference is constant and equal to 2, is of the form $(x+n)^2$, $n=0,\dots,M-1$, for some $x\in R_t$?}\\

Vojta proved that $\mathbf{B^2}(R_t)$ had a positive answer when $R_t$ was the field of meromorphic functions over $\C$, or a function field of characteristic zero. In \cite{PheidasVidaux2} and \cite{PheidasVidaux2bis}, T. Pheidas and the second author used a different method to show that $\mathbf{B^2}(F(t))$ had a negative answer when $F$ had characteristic zero.  The new method was also extendible to  the case of $F$ of positive characteristic.  It turned out that if $F$ had positive characteristic,  $\mathbf{B^2}(F(t))$ had a \emph{negative answer} but one could still derive all the desired logical consequences.

In 1981, D. Hensley (in \cite{Hensley} and \cite{Hensley2}) proved that $\mathbf{B^2}(\mathbb F_p)$ had a positive answer, with $M=p$. This was the first (though as explained above not the last) positive answer to an analogue of B\"uchi's Problem. In the same work, he noticed that a positive answer to $\mathbf{B^2}(\Z)$ is implied by a positive answer to what we now call  Hensley's Problem denoted in the future by $\mathbf{HP^2}(\Z)$ \,:\\

\noindent\emph{Does there exist a positive integer $M$ such that, given any integers $\nu$ and $a$, if the quantity $(\nu+n)^2-a$ is a square for more than $M$ values of $n$ then $a=0$?}\\

\begin{remark}\label{note} This implication is not hard to see.  Indeed, suppose that a sequence $(x_n)$ of integers satisfies
\begin{equation}
\label{Xavier}
x_n^2-2x_{n-1}^2+x_{n-2}^2=2
\end{equation}
for $n=2,\dots,M-1$, namely, the sequence $(x_n^2)$ has constant second difference equal to $2$. In \cite{PheidasVidaux3} it was noted that the quantity $\frac{x_n^2-x_0^2}{n}-n$ does not depend on $n$. Denoting this quantity by $2\nu$, we can now rewrite \eqref{Xavier} as  $x_n^2-x_0^2=2n\nu+n^2$. Therefore we now have
$$
x_n^2-(\nu+n)^2=x_n^2-\nu^2-2n\nu-n^2=x_n^2-\nu^2-(x_n^2-x_0^2)=-\nu^2+x_0^2
$$
which does not depend on $n$. Writing $a=\nu^2-x_0^2$, we obtain $x_n^2=(\nu+n)^2-a$. Hence if $\mathbf{HP^2}(\Z)$ has a positive answer for some $M$, then $a$ must be zero and $\mathbf{B^2}(\Z)$ has a positive answer with the same $M$.
\end{remark}

We might consider the obvious analogues of Hensley's Problem over other rings (over a ring of functions we will ask some $x_n$ to be non-constant). For a general discussion on the equivalence between $\mathbf{B^2}(R)$ and $\mathbf{HP^2}(R)$ (for some rings $R$ the two problems \emph{may not} be equivalent) see the survey \cite{PastenPheidasVidaux}, or \cite{Pasten}.

In \cite{PheidasVidaux1}, T. Pheidas together with the second author proposed a generalization of B\"uchi's Problem to higher powers for a ring $R$, denoted in the future by $\mathbf{B^r}(R)$\,:\\

\noindent\emph{Does there exist a positive integer $M$ such that any sequence of $M$ $r$-th powers in $R$ (not all constant if $R=R_t$ is a ring of functions), whose second difference is constant and equal to $r!$, is of the form $(x+n)^r$, $n=0,\dots,M-1$, for some $x\in R$?}\\

It is easy to see that there is a \emph{Hensley Formulation} of this problem which we denote by $\mathbf{HF^r}(R)$.  More precisely, $\mathbf{B^r}(R)$ is equivalent (over \emph{many} rings) to the following question\,:\\

\noindent\emph{Does there exist a positive integer $M$ such that,  for all  $\nu$, $a_0$, \dots, $a_{r-2}$ in $R$, if the quantity
$$
(\nu+n)^r+a_{r-2}n^{r-2}+\dots+a_{1}n+a_0
$$
is an $r$-th power $x_n^r$ for more than $M$ values of $n$ then $a_0=\dots=a_{r-2}=0$?}\\

Again, if $R$ is a ring of functions, we ask for some $x_n$ to be non-constant. In \cite{PheidasVidaux3}, Pheidas and the second author proved that $\mathbf{HF^3}(F[t])$, hence also $\mathbf{B^3}(F[t])$, had a positive answer with $M=92$, if the field $F$ has characteristic zero.

In \cite{Pasten}, Pasten considered the following problem, called now Hensley's Problem for $r$-th powers and denoted by $\mathbf{HP^r}(R)$\,:\\

\noindent\emph{Does there exist a positive integer $M$ such that, for all  $\nu$ and $a$  in $R$, if the quantity
$$
(\nu+n)^r-a
$$
is an $r$-th power $x_n^r$ for more than $M$ values of $n$ then $a=0$?}\\

As usual, if $R$ is a ring of functions, we ask for some $x_n$ not to be constant. Pasten proved that $\mathbf{HP^r}(F[t])$ had a positive answer if $F$ had characteristic zero, for any $r\geq2$. This result was a new evidence that $\mathbf{B^r}(F[t])$  had a positive answer for any power $r$.

Let $K$ be a function field. In this paper we prove that $\mathbf{HP^r}(K)$ has a positive answer for any $r$ if $K$ has characteristic zero (see Theorem \ref{main} below). This implies in particular that $\mathbf{B^2}(K)$ has a positive answer. We also prove that an analogue of $\mathbf{B^2}(K)$ has a positive answer if $K$ has (large enough) positive characteristic (see Theorem \ref{main2} below) while obtaining all the desired logical consequences as in the case of the rational function fields of positive characteristic.  More specifically we show that while there are non-trivial solutions to B\"{u}chi's equations for large enough $M$, they are of a specific form (these non-trivial solutions were discovered by Pasten, see \cite{PheidasVidaux2bis}).

For both results, the number $M$ depends only on $r$ and the genus of $K$. Note that the dependence on the genus is to be expected\,: if $M$  did not depend on the genus then we could add to $K$ ``enough'' $r$-th powers (while increasing the genus) in order for $(\nu+n)^r-a$ to be an $r$-th power for a few more values of $n$.

In order to state the main theorems we introduce the following notation.

%%%%%%%%%%%%%%%%%%%%%%%%%%%%%%%%%%
\begin{notation}
\begin{enumerate}
\item Let $K$ be a function field of genus $g$ over a field of constants $F$ and let $F_0$ be the prime field of $K$.
\item Let $r\geq 2$ and $M\geq1$ be natural numbers.
\item\label{Mitia} If $\bar c=(c_0,\dots,c_{M-1})$ is a sequence of distinct elements of $F$ and $\xi$ is a primitive $r$-th root of unity, we write
$$
c_{i,j,n}=\frac{c_i-\xi^nc_j}{1-\xi^n}
$$
for any indices $i$ and $j$ and for any $n\in\{1,\dots,r-1\}$.
\item Given $\bar c$ as above, let $\ell(\bar c)$ be equal to $3$ if either $\ch(F)$ does not divide $r$ and for all indices $i,j,k,m,n$ we have $c_{i,j,n} \not =c_{i,k,m}$, or for all indices $i,j,k$ we have
    $$
    [F_0(c_i,c_j,c_k,\xi)\colon F_0(c_i,c_j,c_k)]=r-1
    $$
    (in particular, the latter happens if $\ch(F)=0$ and $c_i$ are rational numbers). Otherwise set $\ell(\bar c)=r+1$.
\item Let
$$
B(r,\ell)=\beta_0(r,\ell)g + \beta_1(r, \ell)
$$
and
$$
\beta_0=\left (8r +4 +\frac{3r}{r-1}\right )r^2\ell, \qquad\textrm{and}\qquad\beta_1=\left (4r +2 + \frac{2r}{r-1}\right )r^2\ell + 1.
$$
\end{enumerate}
\end{notation}
%%%%%%%%%%%%%%%%%%%%%%%%%%%%%%%%%%

%%%%%%%%%%%%%%%%%%%%%%%%%%%%%%%%%%
\begin{theorem}\label{main}
Let $K$ be a function field of genus $g$ over a field of constants $F$ of characteristic $0$. Let $a,\nu\in K$ and $(x_0,\dots, x_{M-1})$ be a sequence of elements of $K$ such that at least one $x_i$ is not in $F$. Let $\bar c=(c_0,\dots,c_{M-1})$ be a sequence of distinct elements of $F$. If $M\geq B(r,\ell)$ and the sequence satisfies
\begin{equation}\label{Gauss}
x_{n}^r=(\nu+c_n)^r-a,\quad n=0,\dots,M-1
\end{equation}
then $a=0$.
\end{theorem}
%%%%%%%%%%%%%%%%%%%%%%%%%%%%%%%%%%

%%%%%%%%%%%%%%%%%%%%%%%%%%%%%%%%%%
\begin{theorem}\label{main2}
Let $K$ be a function field of genus $g$ over a field of constants $F$ of characteristic $p\geq B(2,3)$. Let $a,\nu\in K$ and $(x_0,\dots, x_{M-1})$ be a sequence of elements of $K$ such that at least one $x_i$ is not in $F$. If $M\geq B$, then the sequence satisfies
\begin{equation}\label{Gausss}
x_{n}^2=(\nu+n)^2-a,\, n=0,\dots,M-1
\end{equation}
if and only if, either $a=0$, or there exists a non-negative integer $s$ and $f\in K$ such that for all $n$ we have
\begin{equation}\label{Pasten}
    x_n=(f+n)^{\frac{p^s+1}{2}}.
\end{equation}
\end{theorem}
%%%%%%%%%%%%%%%%%%%%%%%%%%%%%%%%%%

Let $\calL^2_\tau=\calL^2\cup\{\tau\}$ be the language obtained by adding to $\calL^2$ a symbol of unary function $\tau$ for multiplication by a transcendental element $t$ of $K$. Similarly, let $\calL_\tau=\calL\cup\{\tau\}$ be the language obtained by adding to $\calL$ the symbol $\tau$.

In this notation we obtain the following corollaries in Logic\,:

%%%%%%%%%%%%%%%%%%%%%%%%%%%%%%%%%%
\begin{corollary}\label{cor}
If $K$ is a function field of genus $g$ over a field of constants $F$ of characteristic $0$ or $p\geq B(2,3)$, then multiplication over $K$ is positive-existential in the languages $\calL^2_\tau$.
\end{corollary}
%%%%%%%%%%%%%%%%%%%%%%%%%%%%%%%%%%

%%%%%%%%%%%%%%%%%%%%%%%%%%%%%%%%%%
\begin{corollary}\label{cor2}
If $K$ is a function field of genus $g$ over a field of constants $F$ of characteristic $0$ or $p\geq B(2,3)$, then the positive existential theory of $K$ in $\calL^2_\tau$ is undecidable if and only if the positive existential theory of $K$ in $\calL_\tau$ is undecidable.
\end{corollary}
%%%%%%%%%%%%%%%%%%%%%%%%%%%%%%%%%%

There are many function fields for which the positive existential theory is known to be undecidable. For more information, we refer the interested reader to \cite{Ghent1999}, \cite{Poonen} and \cite{Shlapentokh8}.

%%%%%%%%%%%%%%%%%%%%%%%%%%%%%%%%%%
%%%%%%%%%%%%%%%%%%%%%%%%%%%%%%%%%%
%%%%%%%%%%%%%%%%%%%%%%%%%%%%%%%%%%
%%%%%%%%%%%%%%%%%%%%%%%%%%%%%%%%%%
%%%%%%%%%%%%%%%%%%%%%%%%%%%%%%%%%%
%%%%%%%%%%%%%%%%%%%%%%%%%%%%%%%%%%
%%%%%%%%%%%%%%%%%%%%%%%%%%%%%%%%%%
%%%%%%%%%%%%%%%%%%%%%%%%%%%%%%%%%%
%%%%%%%%%%%%%%%%%%%%%%%%%%%%%%%%%%
%%%%%%%%%%%%%%%%%%%%%%%%%%%%%%%%%%
%%%%%%%%%%%%%%%%%%%%%%%%%%%%%%%%%%
%%%%%%%%%%%%%%%%%%%%%%%%%%%%%%%%%%
%%%%%%%%%%%%%%%%%%%%%%%%%%%%%%%%%%
%%%%%%%%%%%%%%%%%%%%%%%%%%%%%%%%%%

\section{Technical preliminaries}

\setcounter{equation}{0}

%%%%%%%%%%%%%%%%%%%%%%%%%%%%%%%%%%
\begin{NotationAssumption}\label{Natasha}
Below we will use the following notation and assumptions.
\begin{enumerate}
\item Let $K$ be a function field of genus $g$ over a field of constants $F$ and let $F_0$ be the prime field of $K$.
\item A \emph{prime} of $K$ is a valuation of $K$.
\item Let $\xi$ be a primitive $r$-th root of unity.
\item If $\mathfrak I$ is an effective (i.e. integral) divisor, we will denote by $\deg \mathfrak I$ the degree of $\mathfrak I$.
\item If $\mathfrak I_1$ and $\mathfrak I_2$ are integral divisors, we write $\mathfrak I_1 | \mathfrak I_2$ ($\mathfrak I_1$ divides $\mathfrak I_2$) to mean that for all primes $\pp$ of $K$ we have that $\ord_{\pp}\mathfrak I_1 \leq \ord_{\pp}\mathfrak I_2$.  Similarly for any prime $\pp$ of $K$ we write that $\pp | \mathfrak I_1$ ($\pp$ divides $\mathfrak I_1$) to mean $\ord_{\pp}\mathfrak I_1 >0$.
\item For $x \in K$, let $\nn(x)$ denote the zero divisor of $x$ and $\dd(x)$ the
pole divisor of $x$. Let $\DD(x)=\frac{\nn(x)}{\dd(x)}$ be the divisor of $x$. Let $\HH(x)$
denote the height of $x$, i.e. $\deg \dd(x) = \deg \nn(x)$.
\item Let $\pp_{\infty}$ be a valuation of $K$.
\item Let $t\in K\setminus F$ having a pole at $\pp_\infty$ only (such a $t$ exists by  \cite[Fried and Jarden, Lemma 3.2.3, p. 55]{FJ}). We can also assume that $t$ is not a $p$-th power in the case $K$ has characteristic $p>0$ (by taking successive $p$-th roots if necessary).
\item For a prime $\pp$ of $K$, let $e(\pp)$ be the ramification degree of $\pp$ over $F(t)$.
\item We can define a global derivation with respect to $t$ as in Mason
 \cite[p. 9]{Mason}. Given an element $x$ of $K$, the derivative with respect to $t$ will be denoted in the usual fashion as $x'$ or $\frac{dx}{dt}$. Observe that usual differentiation rules apply to the global derivation with respect to $t$.  Thus, the only functions with the global derivative with respect to $t$ equal to zero are constants in the case the characteristic is equal to zero and $p$-th powers in the case the characteristic is equal to $p>0$.
\item If the field $F$ is algebraically closed and $\pp$ is a prime of $K$ we can also define a local derivation with respect to the prime $\pp$ as in Mason \cite[p. 9]{Mason}. The derivative of $x\in K$ with respect to $\pp$ will be denoted as $\frac{\partial x}{\partial \pp}$.
\item For all primes $\pp$, let
$$
d(\pp)=\ord_\pp\left(\frac{\partial t}{\partial \pp}\right)
$$
and let
$$
\EE=\prod_{d(\pp)>0}\pp^{d(\pp)}.
$$
\item\label{VectorSpace} If $\mathfrak A$ is a divisor of $K$, we will write
$$
L(\mathfrak A)=\{f\in K \mid \ord_{\pp}f\geq-\ord_\pp\mathfrak A\mbox{ for all primes }\pp \mbox{ of } K\}
$$
and $\ell(\mathfrak A)$ for the dimension of $L(\mathfrak A)$ over $F$.
\item Throughout the paper the following constants will be used\,:
$$
C_1=g+1,\qquad C_2=3g
$$
$$
C_3=C_2+2=3g+2 \qquad\textrm{and}\qquad C_4=C_2+C_1 + 1=4g+2.
$$
\end{enumerate}
\end{NotationAssumption}
%%%%%%%%%%%%%%%%%%%%%%%%%%%%%%%%%%

%%%%%%%%%%%%%%%%%%%%%%%%%%%%%%%%%%
\begin{assumption}
Without loss of generality, we may assume that $F$ is algebraically closed (therefore, all primes of $K$,  in particular $\pp_\infty$, have degree $1$).
\end{assumption}
%%%%%%%%%%%%%%%%%%%%%%%%%%%%%%%%%%

The following lemma gathers some general formulae we need in this section.

%%%%%%%%%%%%%%%%%%%%%%%%%%%%%%%%%%
\begin{lemma}\label{Formulae}
\begin{enumerate}
\item \label{FormulaOrd} Let $E$ be a finite degree subfield of $K$. Let $\PP$ be a prime of $E$ and let $\pp_1,\dots,\pp_n$ be the primes in $K$ above $\PP$. Let $e(\pp_i/\PP)$ be the ramification index of $\pp_i$ over $\PP$. Let $f(\pp_i/\PP)$ be the relative degree of $\pp_i$ over $\PP$ (the degree of the extension of the residue field). We have
$$
[K:E]=\sum_{i=1}^ne(\pp_i/\PP)f(\pp_i/\PP).
$$
\item\label{RiemannRoch} (Riemann-Roch) Let $\mathfrak A$ be a divisor of $K$ of degree $d$.
\begin{enumerate}
\item\label{RR1} If $g=0$ then $\ell(\mathfrak A)= d+1$;
\item\label{RR2} If $g>0$ and $0<d<2g-2$ then $\ell(\mathfrak A)\geq d-g+1$;
\item\label{RR3} If $g>0$ and $d=2g-2$ then $\ell(\mathfrak A)\geq g-1$;
\item\label{RR4} If $g>0$ and $d>2g-2$ then $\ell(\mathfrak A)= d-g+1$;
\end{enumerate}
%\item\label{RiemannHurwitz} (Riemann-Hurwitz) We have
%$$
%\deg\diff(K/F(t))=2g-2+2[K:F(t)],
%$$
%where $\diff(K/F(t))$ stands for the different of $K$ over $F(t)$.
\end{enumerate}
\end{lemma}
%%%%%
\begin{proof}
For \eqref{FormulaOrd} see Fried and Jarden \cite[Proposition 2.3.2, Theorem 3.6.1]{FJ}. For \eqref{RiemannRoch} see Koch \cite[Theorem 5.6.2]{Koch}.
\end{proof}
%%%%%%%%%%%%%%%%%%%%%%%%%%%%%%%%%%

%%%%%%%%%%%%%%%%%%%%%%%%%%%%%%%%%%
\begin{lemma}\label{RR}
If $\mathfrak A$ is a divisor of $K$ of degree $g+1$ then $\ell(\mathfrak A)\geq2$.
\end{lemma}
%%%%%
\begin{proof}
Since $\mathfrak A$ has degree $d=g+1$, we have
\begin{itemize}
\item if $g=0$ then $d=1$ and $\ell(\mathfrak A)=d+1=2$ by Lemma \ref{Formulae} \eqref{RR1};
\item if $g=1$ or $2$ then $d>2g-2$ and $\ell(\mathfrak A)=d-g+1=2$ by Lemma \ref{Formulae} \eqref{RR4};
\item if $g=3$ then $d=4=2g-2$ and $\ell(\mathfrak A)\geq g-1=2$ by Lemma \ref{Formulae} \eqref{RR3};
\item if $g\geq4$ then $d=g+1<2g-2$ and $\ell(\mathfrak A)\geq d-g+1=2$ by Lemma \ref{Formulae} \eqref{RR2}.
\end{itemize}
Hence in all cases, $\ell(\mathfrak A)\geq2$.
\end{proof}
%%%%%%%%%%%%%%%%%%%%%%%%%%%%%%%%%%

%%%%%%%%%%%%%%%%%%%%%%%%%%%%%%%%%%
\begin{lemma}\label{Mason}
Let $x\in K$ and $\pp$ be a prime of $K$. We have
\begin{enumerate}
\item\label{Mason1} $\ord_\pp(\frac{\partial x}{\partial \pp})\geq\ord_\pp(x)-1$; and
\item\label{Mason2} if $\ord_\pp(x)\geq0$, then $\ord_\pp(\frac{\partial x}{\partial \pp})\geq0$.
\end{enumerate}
\end{lemma}
\begin{proof}
See Mason \cite[p. 9]{Mason}.
\end{proof}
%%%%%%%%%%%%%%%%%%%%%%%%%%%%%%%%%%

%%%%%%%%%%%%%%%%%%%%%%%%%%%%%%%%%%
\begin{lemma}\label{Luba}
The function $t$ can be chosen so that
\begin{enumerate}
\item\label{exte} $[K:F(t)]\leq C_1$,
\item\label{dpnotinf} $d(\pp)\geq0$ for all $\pp\ne\pp_\infty$,
\item\label{dpinf} $d(\pp_\infty)\geq-g-2$, and
\item\label{deinf} $\deg\EE \le C_2$.
\end{enumerate}
\end{lemma}
%%%%%
\begin{proof}
Since the integral divisor $\pp_\infty^{g+1}$ of $K$ has degree $g+1$, we have
$$
\ell\left(\pp_\infty^{g+1}\right)=2>1
$$
by Lemma \ref{RR}. Therefore, $L(\pp_\infty^{g+1})$ contains a non-constant element $w$ such that
$$
\dd(w)=\pp_\infty^\alpha,
$$
where $\alpha\leq g+1$. Let us show that $w$ satisfies the conclusions of the lemma.
\begin{enumerate}
\item Let $\mathfrak P_\infty$ be the prime of $F(w)$ below $\pp_\infty$. Observe that the ramification degree of $\pp_\infty$ over $\mathfrak P_\infty$ is $\alpha$ and each prime has degree $1$ in its respective field. Since there is no constant field extension we also conclude that the relative degree of $\pp_\infty$ over $\mathfrak P_\infty$ is $1$. Thus by Lemma \ref{Formulae} \eqref{FormulaOrd} we have $[K:F(w)]=\alpha\leq g+1$ and we can choose $w$ as our new $t$.  If $p=\ch(K) >0$ and $w$ happens to be a $p$-th power, we will replace $w$ by its $p$-th root sufficiently many times until the result is no longer a $p$-th power in $K$.  Observe that taking a $p$-th root will only reduce $\alpha$, and therefore the conclusion of the lemma remains unchanged.  Observe also that we can assume that $dw/dt \not =0$. For the rest of the proof, let $d_w(\pp)$ stands for $\ord_{\pp}\left(\frac{\partial w}{\partial\pp}\right)$. 
    %instead of $\ord_{\pp_{\infty}}\left(\frac{\partial t}{\partial\pp_{\infty}}\right)$.
\item By Lemma \ref{Mason} \eqref{Mason2} we have that $d_w(\pp)\geq0$ for all $\pp\ne\pp_\infty$,
\item By Lemma \ref{Mason} \eqref{Mason1}, we have $d_w(\pp_\infty)\geq-\alpha-1\geq-g-2$.
\item By Mason \cite[Equation (5) p. 10]{Mason}, we have
$$
\sum_{\pp} d_w(\pp)=\sum_{\pp}\ord_\pp\left(\frac{\partial w}{\partial\pp}\right)=2g-2
$$
since $w$ has non-zero global derivative.  Therefore, by Items \eqref{dpnotinf} and \eqref{dpinf}, we have
$$\sum_{d_w(\pp)>0} d_w(\pp) \leq 2g-2 < 3g, $$
if $\ord_{\pp_{\infty}}\left(\frac{\partial w}{\partial\pp_{\infty}}\right)\geq 0,$
and
$$\sum_{d_w(\pp)>0} d_w(\pp) = 2g-2 -\ord_{\pp_{\infty}}\left(\frac{\partial w}{\partial\pp_{\infty}}\right)\leq 2g-2+g+2=3g,$$
if $\ord_{\pp_{\infty}}\left(\frac{\partial w}{\partial\pp_{\infty}}\right)< 0$.
\end{enumerate}
\end{proof}
%%%%%%%%%%%%%%%%%%%%%%%%%%%%%%%%%%

%%%%%%%%%%%%%%%%%%%%%%%%%%%%%%%%%%
\begin{lemma}\label{Ania}
For all $x\in K$ and $\pp$ prime of $K$, we have
\begin{enumerate}
\item\label{gaga1} if $\ord_\pp(x)\geq0$ then
$$
\ord_\pp(x')\ge\max(0,\ord_\pp(x)-1) -d(\pp)
$$
and
\item\label{gaga2} if $\ord_\pp(x)<0$ then
$$
\ord_\pp(x')\ge\ord_\pp(x)-1-d(\pp).
$$
\end{enumerate}
\end{lemma}
%%%%%%
\begin{proof}
From Mason \cite[p. 96]{Mason} we have for any prime $\pp$ (including $\pp_{\infty}$)
\begin{equation}\label{Michal}
\frac{\partial x}{\partial \pp}=\frac{dx}{dt}\frac{\partial t}{\partial \pp}
\end{equation}
hence, if $\ord_\pp(x)\geq0$ then
$$
\ord_\pp(x')=\ord_\pp\left(\frac{dx}{dt}\right)=\ord_\pp\left(\frac{\partial x}{\partial \pp}\right)-\ord_\pp\left(\frac{\partial t}{\partial \pp}\right)\geq\max(0,\ord_\pp(x)-1)-d(\pp)
$$
and if $\ord_\pp(x)<0$ then
$$
\ord_\pp(x')=\ord_\pp\left(\frac{dx}{dt}\right)=\ord_\pp\left(\frac{\partial x}{\partial \pp}\right)-\ord_\pp\left(\frac{\partial t}{\partial \pp}\right)\geq\ord_\pp(x)-1-d(\pp)
$$
by Lemma \ref{Mason}.
\end{proof}
%%%%%%%%%%%%%%%%%%%%%%%%%%%%%%%%%%

%%%%%%%%%%%%%%%%%%%%%%%%%%%%%%%%%%
\begin{corollary}\label{Elodie}
\begin{enumerate}
\item\label{Elodie1} Let $x$ be a non constant element of $K$. If $\pp$ is a prime of $K$ such that $\ord_\pp(x)\geq0$ and $\ord_\pp(x')<0$, then $d(\pp)>0$ (so that  $\pp | \EE)$, and we have
$$
\ord_\pp(x')\geq-d(\pp).
$$
\item\label{Raisa} If $x$ is a non constant element of $K$ then $\dd(x')$ divides $\dd(x^2)\EE$.
%\item\label{Elodie3} If $x,y\in K$ have the same pole divisor, then the pole divisors of their derivatives $\frac{dx}{dt}$ and $\frac{dy}{dt}$ can differ only by a factor whose degree is %bounded by the constant $C_2$.
\end{enumerate}
\end{corollary}
%%%%%%
\begin{proof}
\begin{enumerate}
\item By Lemma \ref{Ania} Item \eqref{gaga1}, for $x$ without a pole at $\pp$ we have
$$
0>\ord_\pp(x')\geq\max(0,\ord_\pp(x)-1)-d(\pp)\geq-d(\pp).
$$
\item If $\pp$ is a pole of $x'$ then
\begin{itemize}
\item either it does not divide $\EE$ (hence $d(\pp)\leq0$), in which case it is a pole of $x$ (by Item \eqref{Elodie1}), and we have $\ord_\pp(x')\geq\ord_\pp(x)-1-d(\pp)\geq \ord_\pp(x)-1$ by Lemma \ref{Ania} \eqref{gaga2}, hence
$$
\ord_\pp(\dd(x'))\leq\ord_\pp(\dd(x))+1;
$$
\item or it divides $\EE$ (hence $d(\pp)>0$), in which case
	\begin{itemize}
	\item either $\ord_\pp(x)<0$, hence $\ord_\pp(x')\geq\ord_\pp(x)-1-d(\pp)$ by Lemma \ref{Ania} \eqref{gaga2}, and we conclude
	$$
	\ord_\pp(\dd(x'))\leq\ord_\pp(\dd(x))+1+d(\pp);
	$$
	\item or $\ord_\pp(x)\geq 0$, hence $\ord_\pp(x')\geq -d(\pp)$ (by Item \eqref{Elodie1}),
	hence
	$$
	\ord_\pp(\dd(x'))\leq d(\pp).
	$$
	\end{itemize}
\end{itemize}
We deduce that $\dd(x')$ divides
$$
\prod_{\pp\nmid\EE}{\pp^{\ord_\pp(\dd(x))+1}}
\prod_{\substack{\pp\mid\EE\\\ord_\pp(\dd(x))>0}}{\pp^{\ord_\pp(\dd(x))+1+d(\pp)}}
\prod_{\substack{\pp\mid\EE\\\ord_\pp(\dd(x))=0}}{\pp^{d(\pp)}}
$$
where in the first product we have $\ord_\pp(\dd(x))>0$. Multiplying the rightmost product by
$$
\prod_{\substack{\pp\mid\EE\\\ord_\pp(\dd(x))>0}}{\pp^{d(\pp)}}
$$
and dividing the `middle product' by the same quantity, we see that $\dd(x')$ divides
$$
\prod_{\substack{\pp\nmid\EE\\\ord_\pp(\dd(x))>0}}{\pp^{\ord_\pp(\dd(x))+1}}
\prod_{\substack{\pp\mid\EE\\\ord_\pp(\dd(x))>0}}{\pp^{\ord_\pp(\dd(x))+1}}
\prod_{\pp\mid\EE}{\pp^{d(\pp)}}
$$
which in turn divides
$$
\prod_{\pp}\pp^{2\ord_\pp(\dd(x))}\prod_{\pp\mid\EE}{\pp^{d(\pp)}}
=\dd(x^2)\EE
$$
which was to be proved.

%\item By Item (1) of this lemma the only primes that can occur in the pole divisors of $\frac{dx}{dt}$ and $\frac{dy}{dt}$ without being  poles of $x$ and $y$, are the primes which %divide $\EE$. Further, by Item (1) of this lemma, in this case they can occur to the power at most equal to $d(\pp)$. Therefore, the degree of the product of these extra factors will %be bounded by $\deg (\EE)=C_2$ (see Lemma \ref{Luba} for the last inequality).
\end{enumerate}
\end{proof}
%%%%%%%%%%%%%%%%%%%%%%%%%%%%%%%%%%

%%%%%%%%%%%%%%%%%%%%%%%%%%%%%%%%%%
\begin{corollary}\label{Katia}
For any $x\in K$ which is not a constant, we have
$$
\deg\dd(x')\leq C_3\deg\dd (x).
$$
\end{corollary}
%%%%%
\begin{proof}
From Corollary \ref{Elodie} \eqref{Raisa}  we have that $\dd(x')$ divides $\dd(x^2)\EE$ and therefore
$$
\deg(\dd(x')) \leq \deg(\dd(x^2)\EE) \leq 2\deg \dd(x) + C_2 \leq (C_2+2)\deg\dd(x)=C_3\deg\dd(x)
$$
by Lemma \ref{Luba} \eqref{deinf} and definition of $C_3$.
\end{proof}
%%%%%%%%%%%%%%%%%%%%%%%%%%%%%%%%%%

%%%%%%%%%%%%%%%%%%%%%%%%%%%%%%%%%%
\begin{lemma}\label{Nicole}
For any non-trivial effective divisor $\mathfrak A$ there exists $y\in K$ such that
\begin{enumerate}
\item the divisor $\mathfrak A\EE$ divides $\nn(y)$;
\item the function $y$ has only one pole at $\pp_{\infty}$; and
\item we have
$$
\deg\dd(y) \leq C_4\deg(\mathfrak A).
$$
\end{enumerate}
\end{lemma}
%%%%%%
\begin{proof}
Let
$$
\mathfrak B=\frac{\mathfrak A\EE}{\pp_\infty^d}
$$
where $d=\deg(\mathfrak A\EE)+g+1$. Since $\mathfrak B^{-1}$ has degree $g+1$, we have
$$
\ell(\mathfrak B^{-1})\geq2>1
$$
by Lemma \ref{RR}. Therefore, the vector space $L(\mathfrak B^{-1})$ contains a non-constant element $y$ such that $\dd(y)=\pp_\infty^\alpha$ where $1\leq\alpha\leq d$, and $\nn(y)$ is divisible by $\mathfrak A\EE$, so that Items (1) and (2) are satisfied. Finally observe that
$$
\deg(\dd(y))=\alpha\leq d=\deg(\mathfrak A\EE)+g+1=\deg(\mathfrak A)+\deg(\EE)+g+1\leq \deg(\mathfrak A)+C_2+C_1,
$$
where the last inequality holds by Lemma \ref{Luba}. We finally get
$$
\deg(\dd(y))\leq\deg(\mathfrak A)+C_2+C_1
\leq (C_2+C_1+1)\deg(\mathfrak A)=C_4\deg(\mathfrak A),
$$
where the last inequality comes from the fact that $\deg\mathfrak A\geq1$.
\end{proof}
%%%%%%%%%%%%%%%%%%%%%%%%%%%%%%%%%%

%%%%%%%%%%%%%%%%%%%%%%%%%%%%%%%%%%
%%%%%%%%%%%%%%%%%%%%%%%%%%%%%%%%%%
%%%%%%%%%%%%%%%%%%%%%%%%%%%%%%%%%%
%%%%%%%%%%%%%%%%%%%%%%%%%%%%%%%%%%
%%%%%%%%%%%%%%%%%%%%%%%%%%%%%%%%%%
%%%%%%%%%%%%%%%%%%%%%%%%%%%%%%%%%%
%%%%%%%%%%%%%%%%%%%%%%%%%%%%%%%%%%
%%%%%%%%%%%%%%%%%%%%%%%%%%%%%%%%%%
%%%%%%%%%%%%%%%%%%%%%%%%%%%%%%%%%%
%%%%%%%%%%%%%%%%%%%%%%%%%%%%%%%%%%
%%%%%%%%%%%%%%%%%%%%%%%%%%%%%%%%%%
%%%%%%%%%%%%%%%%%%%%%%%%%%%%%%%%%%

\section{Intermediate Theorem}
\setcounter{equation}{0}

This section is devoted to the proof of Theorem \ref{Savka} below. In order to state the theorem we introduce the following notation.

%%%%%%%%%%%%%%%%%%%%%%%%%%%%%%%%%%
\begin{notation}
\begin{enumerate}
\item Let $r\geq 2$ and $M$ be positive natural numbers.
\item If $\bar c=(c_0,\dots,c_{M-1})$ is a sequence of distinct elements of $F$, we will write
$$
c_{i,j,n}=\frac{c_i-\xi^nc_j}{1-\xi^n}
$$
for any indices $i$ and $j$ and for any $n\in\{1,\dots,r-1\}$.

\item Given $\bar c$ as above, let $\ell(\bar c)$ be equal to $3$ if for all indices $i,j,k,m,n$ we have either $c_{i,j,n} \not =c_{i,k,m}$, or for all indices $i,j,k$ we have
    $$
    [F_0(c_i,c_j,c_k,\xi)\colon F_0(c_i,c_j,c_k)]=r-1
    $$
    (in particular, the latter happens if $\ch(F)=0$ and $c_i$ are rational numbers). Otherwise set $\ell(\bar c)=r+1$.
\end{enumerate}
\end{notation}
%%%%%%%%%%%%%%%%%%%%%%%%%%%%%%%%%%

%%%%%%%%%%%%%%%%%%%%%%%%%%%%%%%%%%
\begin{theorem}\label{Savka}
Let $a,\nu\in K$ and $(x_0,\dots, x_{M-1})$ be a sequence of elements of $K$ such that at least one $x_i$ has non-zero derivative. Let $\bar c=(c_0,\dots,c_{M-1})$ be a sequence of distinct elements of $F$. If
$$
M\geq r^2\ell(\bar c)\left(1+C_3\left(5r+\frac{r+1}{r-1}\right)\right)+1
$$
and
$$
x_{n}^r=(\nu+c_n)^r-a,\, n=0,\dots,M-1
$$
then either $a=0$ or there exist $\gamma\in K$ such that $\gamma'=0$, and $\xi_0$ an $r$-th root of unity, such that
$$
a=(\xi_0\nu+\gamma)^r.
$$
\end{theorem}
%%%%%%%%%%%%%%%%%%%%%%%%%%%%%%%%%%

Throughout this section we will suppose that $a,\nu,x_0,\dots,x_{M-1}$ and $\bar c$ satisfy the hypothesis of Theorem \ref{Savka}.

The following notation will also be used throughout the section.

%%%%%%%%%%%%%%%%%%%%%%%%%%%%%%%%%%
\begin{notation}\label{Groth}
\begin{enumerate}
\item Write $u_n=x_n^r$ and
$$
\DD=\prod_{i=0}^{M-1}\dd(x_i)\qquad\textrm{and}\qquad
\NN=\prod_{i=0}^{M-1}\nn(x_i).
$$
\item Let $d= \deg\DD$.
\item Let $\LL_\dd=\LCM(\dd(x_0),\ldots,\dd(x_{n-1}))$ (where $\LCM$ stands for ``the least common multiple''). Let $\LL_\nn=\LCM(\nn(x_0),\dots,\nn(x_{n-1}))$.
\item Let $y\in K$ be such that
\begin{itemize}
\item the divisor $\mathfrak \LL_\dd\EE$ divides $\nn(y)$;
\item the function $y$ has only one pole at $\pp_{\infty}$; and
\item\label{Charlotte} $\deg\dd(y) \leq C_4\deg\LL_\dd$
\end{itemize}
(such a $y$ exists by Lemma \ref{Nicole} and because $\deg\LL_\dd\ne0$ since by hypothesis at least one $x_i$ is non-constant).
\end{enumerate}
\end{notation}
%%%%%%%%%%%%%%%%%%%%%%%%%%%%%%%%%%

%%%%%%%%%%%%%%%%%%%%%%%%%%%%%%%%%%
\begin{remark}\label{pinf}{\rm
In the previous section there was no assumption whatsoever on $\pp_{\infty}$. We will now set it to be a valuation of $K$ \emph{not occurring as a pole or zero of any element of the (finite) set $\{x_0,\dots,x_{M-1},\nu,a\}$}.}
\end{remark}
%%%%%%%%%%%%%%%%%%%%%%%%%%%%%%%%%%

%%%%%%%%%%%%%%%%%%%%%%%%%%%%%%%%%%
\begin{lemma}\label{LemmaVolodia}
The following equality holds\,:
\begin{equation}\label{Volodia}
u_i-u_j=r(c_i-c_j)\prod_{n=1}^{r-1}\left[\nu+c_{i,j,n}\right]
\end{equation}
where $c_{i,j,n}$ have been defined in Notation \ref{Groth} \eqref{Mitia}.
\end{lemma}
%%%%%%
\begin{proof}
From Equation \eqref{Gauss}, we have
$$
\begin{aligned}
u_i-u_j&=(\nu+c_i)^r-(\nu+c_j)^r\\
&=(c_i-c_j)\prod_{n=1}^{r-1}[(\nu+c_i)-\xi^n(\nu+c_j)]\\
&=(c_i-c_j)\prod_{n=1}^{r-1}[(1-\xi^n)\nu+(c_i-\xi^nc_j)]\\
&=(c_i-c_j)\prod_{n=1}^{r-1}(1-\xi^n)\prod_{n=1}^{r-1}\left[\nu+\frac{c_i-\xi^nc_j}{1-\xi^n}\right]
\end{aligned}
$$
hence
$$
u_i-u_j=r(c_i-c_j)\prod_{n=1}^{r-1}\left[\nu+c_{i,j,n}\right].
$$
%The polynomial
%$$
%1+X+\dots+X^{r-1}=\frac{X^r-1}{X-1}=(X-\xi)\dots(X-\xi^{r-1})
%$$
%evaluated at $1$ gives $r=(1-\xi)\dots(1-\xi^{r-1})$.
\end{proof}
%%%%%%%%%%%%%%%%%%%%%%%%%%%%%%%%%%

%%%%%%%%%%%%%%%%%%%%%%%%%%%%%%%%%%
\begin{lemma}\label{Stelios}
\begin{enumerate}
\item At most one $x_i$ is an element of $F$.
\item For any prime $\pp$ of $K$, either
\begin{equation}\label{Vera}
\ord_\pp \dd(u_n)=\ord_\pp \dd(u_m)\geq(r-1)\ord_\pp\dd(\nu)
\end{equation}
for all $m$ and $n$, or there exists $n_0=n_0(\pp)$ such that
\begin{equation}\label{Chaim}
(r-1)\ord_\pp\dd(\nu)=\ord_\pp\dd(u_n)>\ord_\pp \dd(u_{n_0})
\end{equation}
for all $n$ distinct from $n_0$.
\end{enumerate}%
\end{lemma}%
%%%%%%
\begin{proof}
\begin{enumerate}
\item Fix an index $k$ and suppose that $x_k$ is not constant (we know by hypothesis of Theorem \ref{Savka} that there exists at least one such $k$). Suppose that there exists an index $i \not = k$ such that $x_i$ is constant. From Equation \eqref{Volodia}, substituting $k$ for $j$, it follows that $\nu$ is not a constant. Hence for any $j\ne i$, Equation \eqref{Volodia} for $i$ and $j$ implies that $x_j$ is not a constant.
\item Fix a prime $\pp$ of $K$. If for all indices $n$ and $m$ we have $\ord_\pp \dd(u_n)=\ord_\pp \dd(u_m)$, then by Equation \eqref{Volodia} for $n$ and $m$, we have also
$$
\ord_\pp\dd(u_m)\geq(r-1)\ord_\pp\dd(\nu).
$$
Hence \eqref{Vera} holds. Otherwise there exist indices $n_0\ne n_1$ such that for
$$
\ord_\pp \dd(u_{n_0})<\ord_\pp\dd(u_{n_1}).
$$
From Equation \eqref{Volodia} with indices $n_0$ and $n_1$, we have
$$
(r-1)\ord_\pp\dd(\nu)=\ord_\pp\dd(u_{n_1}).
$$
From the same equation for any index $j\ne n_0$ and $n_0$ we conclude that
$$
\ord_\pp\dd(u_j)=(r-1)\ord_\pp\dd(\nu).
$$
Hence \eqref{Chaim} holds.
\end{enumerate}
\end{proof}
%%%%%%%%%%%%%%%%%%%%%%%%%%%%%%%%%%

%%%%%%%%%%%%%%%%%%%%%%%%%%%%%%%%%%
\begin{proposition}\label{yheight}
The following inequalities hold (see Notation \ref{Natasha})\,:
\begin{enumerate}%
\item for any index $n$ and prime $\pp$ of $K$
$$
\ord_\pp \dd (x_n)\leq\frac{\ord_\pp\DD}{M-1};
$$
\item\label{deglcm}  $\displaystyle \deg \LL_\dd \leq \frac{d}{M-1}$;
\item $\displaystyle \deg \dd(x_n) \leq \frac{d}{M-1}$ for any index $n$.
\end{enumerate}%
\end{proposition}%
%%%%%
\begin{proof}%
\begin{enumerate}
\item By Lemma \ref{Stelios}, we have
$$
\ord_{\pp}[\dd(x_0)\dots\dd(x_{n})\dots\dd(x_{M-1})] \geq (M-1)\ord_{\pp}\dd(x_n)
$$
for any prime $\pp$ in $K$ and any index $n$ (note that we consider the product of $M$ factors on the left-hand side).
\item For any prime $\pp$ such that $\ord_\pp(\LL_\dd)>0$, by definition of $\LL_\dd$ there exists an index $n$ such that $\ord_\pp(\LL_\dd)=\ord_\pp(\dd(x_n))$ and therefore by Item (1) we have
    $$
    \ord_\pp(\LL_\dd)\leq\frac{\ord_\pp\DD}{M-1}.
    $$\
\item This part follows directly from either (1) or (2).
\end{enumerate}
\end{proof}%
%%%%%%%%%%%%%%%%%%%%%%%%%%%%%%%%%%

%%%%%%%%%%%%%%%%%%%%%%%%%%%%%%%%%%
\begin{lemma}\label{nuheight}%
The following inequalities hold\,:
\begin{enumerate}
\item\label{ordnu} For any prime $\pp$ of $K$ and for all but at most one index $n$ we have
$$
(r-1)\ord_\pp\dd(\nu)\leq r\ord_\pp\dd(x_n).
$$
\item We have
$$
(r-1)(M-1)\deg\dd(\nu) \leq rd.
$$
\item\label{orda} For any prime $\pp$ of $K$ and for all but at most one index $n$ we have
$$
\ord_\pp\dd(a)\leq r\ord_\pp\dd(x_n).
$$
\item We have
$$
(M-1)\deg\dd(a) \leq rd.
$$
\end{enumerate}
\end{lemma}%
%%%%%%
\begin{proof}%
\begin{enumerate}
\item This comes from Lemma \ref{Stelios} and by definition of $u_n=x_n^r$.
\item By Proposition \ref{yheight} we have
$$
\ord_\pp\dd(x_n)\leq \frac{\ord_\pp\DD}{M-1}
$$
for all $n$. From Item (1) we deduce
$$
(r-1)\ord_\pp\dd(\nu)\leq\frac{r}{M-1}\ord_\pp\DD
$$
and the claim follows.
\item From Equation \eqref{Gauss} we have for any index $n$ and any prime $\pp$ in $K$
$$
\ord_\pp\dd(a)\leq\max\{r\ord_\pp\dd(x_n),r\ord_\pp\dd(\nu)\}.
$$
Hence by Item (1),  for all but at most one index $n$, we have
$$
\begin{aligned}
\ord_\pp\dd(a)&\leq \max\{r\ord_\pp\dd(x_n),\frac{r}{r-1}\ord_\pp\dd(x_n)\}\\
&\leq r\ord_\pp\dd(x_n)
\end{aligned}
$$
which was to be proved.
\item From Item (3) and Proposition \ref{yheight} we have
$$
\ord_\pp\dd(a)\leq r\ord_\pp\dd(x_n) \leq \frac{r \ord_\pp\DD}{M-1},
$$
hence
$$
\deg(\dd(a))\leq\frac{rd}{M-1}
$$
by definition of $d$.
\end{enumerate}
\end{proof}%
%%%%%%%%%%%%%%%%%%%%%%%%%%%%%%%%%%

%%%%%%%%%%%%%%%%%%%%%%%%%%%%%%%%%%
\begin{corollary}\label{Nicolas}
The divisors $\dd(x_n)$, $\dd(a)$, $\dd(\nu)$, $\dd(x_n')$, $\dd(a')$ and $\dd(\nu')$ divide $\nn(y^{2r+1})$. Moreover we have
$$
\deg\dd(y)\leq C_4\frac{d}{M-1}.
$$
\end{corollary}%
%%%%%%
\begin{proof}
Recall that $y\in K$ is such that the divisor $\LL_\dd\EE$ divides $\nn(y)$
%, the function $y$ has only one pole at $\pp_{\infty}$ and $\deg\dd(y) \leq C_4\deg \LL_\dd$
(see Notation \ref{Groth}), hence for all primes $\pp\in K$ and for all index $n$, we have
\begin{equation}\label{Maestro}
\ord_\pp(\dd(x_n))\leq\ord_\pp(\LL_\dd)\leq\ord_\pp(\nn(y)).
\end{equation}
(reacall that $\LL_\dd$ is the least common multiple of the $\dd(x_n)$).

Also, since $\deg\dd(y) \leq C_4\deg\LL_\dd$ (see Notation \ref{Groth}), by Proposition \ref{yheight} \eqref{deglcm} we get
$$
\deg\dd(y)\leq C_4\frac{d}{M-1}.
$$
\begin{enumerate}
\item Equation \eqref{Maestro} implies that $\dd(x_n)$ divides $\nn(y)$.
\item From Lemma \ref{nuheight} \eqref{orda}, for all prime $p$ of $K$ we have $\ord_\pp\dd(a)\leq r\ord_\pp\dd(x_n)$, hence $\ord_\pp\dd(a)\leq r\ord_\pp\nn(y)$ by Equation \eqref{Maestro}. So $\dd(a)$ divides $\nn(y^r)$.
\item From Lemma \ref{nuheight} \eqref{ordnu}, we have
$$
(r-1)\ord_\pp\dd(\nu)\leq r\ord_\pp\dd(x_n),
$$
hence
$$
\ord_\pp\dd(\nu)\leq\frac{r}{r-1}\ord_\pp\dd(x_n)\leq\frac{r}{r-1}\ord_\pp y\leq(2r+1)\ord_\pp y
$$
by Equation \eqref{Maestro}. Therefore, $\dd(\nu)$ divides $\nn(y^{2r+1})$.
\item By Corollary \ref{Elodie} \eqref{Raisa}, the pole divisor of $x_n'$ divides $\dd(x_n^2)\EE$, which in turn divides $\nn(y^{2+1})$. Observe that $3$ is less than $2r+1$.
\item By Corollary \ref{Elodie} \eqref{Raisa} again, the pole divisor of $a'$ divides $\dd(a^2)\EE$, which in turn divides $\nn(y^{2r})\nn(y)$ by (2) and because $\EE$ divides $\nn(y)$. Hence $\dd(a')$ divides $\nn(y^{2r+1})$.
\item Similarly, by Corollary \ref{Elodie} \eqref{Raisa}, the pole divisor of $\nu'$ divides $\dd(\nu^2)\EE$, hence by Item (3)
$$
\ord_\pp\dd(\nu')\leq\ord_\pp\dd(\nu^2\EE)\leq
\left(2\frac{r}{r-1}+1\right)\ord_\pp y\leq(2r+1)\ord_\pp y
$$
and we conclude that $\dd(\nu')$ divides $\nn(y^{2r+1})$.
\end{enumerate}
\end{proof}
%%%%%%%%%%%%%%%%%%%%%%%%%%%%%%%%%%

%%%%%%%%%%%%%%%%%%%%%%%%%%%%%%%%%%
\begin{lemma}\label{Paul}
\begin{enumerate}
\item For any set of distinct indices $n_1,\dots,n_{r+1}$, the functions $x_{n_i}$ do not have a common zero.
\item If the characteristic $p$ of $K$ does not divide $r$ and for all indices $i,j,k,m,n$ we have that $c_{i,j,n} \not =c_{i,k,m}$, then for any three distinct indices $i$, $j$ and $k$, the functions $x_i$, $x_j$ and $x_k$ do not have a common zero.
\item Suppose $r\ne2$. If for all indices $i, j,k$ we have that
$$
[F_0(c_i,c_j,c_k,\xi)\colon F_0(c_i,c_j,c_k)]=r-1
$$
then for all distinct indices $i, j, k$ we have that  $x_i$, $x_j$ and $x_k$ do not have a common zero. In particular, this is true if $c_i$ are rational numbers.
\end{enumerate}
\end{lemma}
%%%%%%
\begin{proof}
\begin{enumerate}
\item Since we have assumed that the field of constants is algebraically closed, the proof in Pasten \cite[Lemma 3.3]{Pasten} goes through for the general case essentially unchanged.
\item From \eqref{Volodia}, we have for any indices $i$ and $j$
$$
x_i^r-x_j^r=u_i-u_j=r(c_i-c_j)\prod_{n=1}^{r-1}\left[\nu+c_{i,j,n}\right].
$$
Suppose now that $\pp$ is a prime of $K$ which is a common zero of $x_i$, $x_j$ and $x_k$ for some distinct indices $i$, $j$ and $k$. Consequently, for some $n$ and $m$, $\pp$ is a zero of $\nu+c_{i,j,n}$ and $\nu+c_{i,k,m}$, hence of $c_{i,j,n}-c_{i,k,m}\in F$, implying
\begin{equation}\label{Philip}
c_{i,j,n}=c_{i,k,m}.
\end{equation}
\item Suppose that Equation \eqref{Philip} holds for some $i,j,k,m$ and $n$. Without loss of generality, assume $n\geq m$. By definition of $c_{i,j,n}$, from Equation \eqref{Philip} we get
$$
\begin{aligned}
0&=(1-\xi^m)(c_i-\xi^nc_j)-(1-\xi^n)(c_i-\xi^mc_k)\\
&=\xi^m(c_k-c_i)+\xi^n(c_i-c_j)+\xi^{n+m}(c_j-c_k)\\
&=\xi^m[(c_k-c_i)+\xi^{n-m}(c_i-c_j)+\xi^n(c_j-c_k)].
\end{aligned}
$$
If $n=m$, since $c_j\ne c_k$ (by hypothesis of Theorem \ref{Savka}), then $\xi^n=1$, which is impossible since $$1\leq n\leq r-1.$$ Otherwise, $1$, $\xi^{n-m}$ and $\xi^n$ are linearly dependent over $F_0(c_i,c_j,c_k)$. This  contradicts our assumption on the degree of the extension if $r>3$.  If $r=3$, since $n >m$,  we must have that $n=2$ and $m=1$, yielding
$$
\begin{aligned}
0&=(c_k-c_i)+\xi(c_i-c_j)+\xi^2(c_j-c_k)\\
&=(c_k-c_i)+\xi(c_i-c_j)-(1+\xi)(c_j-c_k)\\
&=(2c_k-c_i-c_j)+\xi(c_i-2c_j+c_k).
\end{aligned}
$$
The last equation under our assumptions is equivalent to the system
$$
\left \{
\begin{aligned}
2c_k-c_i-c_j=0\\
c_i-2c_j+c_k=0\\
\end{aligned}
\right .
$$
Replacing $c_i$ in the first equation by $2c_j-c_k$ we obtain $2c_k-2c_j +c_k -c_j=0$, i.e. $c_k=c_j$ contradicting our assumptions on $\bar c$ in the hypothesis of Theorem \ref{Savka}.
Therefore the assumption of Item (2) holds.
\end{enumerate}
\end{proof}
%%%%%%%%%%%%%%%%%%%%%%%%%%%%%%%%%%

%%%%%%%%%%%%%%%%%%%%%%%%%%%%%%%%%%
\begin{notation}
Let $\ell\geq2$ be a natural number such that any $\ell$ of the $x_i$ are coprime (such an $\ell$ exists by Lemma \ref{Paul}).
\end{notation}
%%%%%%%%%%%%%%%%%%%%%%%%%%%%%%%%%%

Recall that $\LL_\nn$ stands for the least common multiple of the $\nn(x_n)$ (see Notation \ref{Groth}).

%%%%%%%%%%%%%%%%%%%%%%%%%%%%%%%%%%
\begin{corollary}\label{Francois}
The following inequality holds\,:
$$
\deg\LL_\nn\geq\frac{d}{\ell}\ .
$$
\end{corollary}%
%%%%%%
\begin{proof}
Let $\pp$ be a prime such that $\ord_\pp\NN>0$ (where $\NN$ is the product of the numerator divisors of the $x_n$). Further, let $x_{i_1}$, \dots, $x_{i_s}$, with $s<\ell$, be all the functions in the sequence $(x_i)$ with a zero at $\pp$. Without loss of generality, assume $$
\ord_\pp x_{i_1}\geq\dots\geq\ord_\pp x_{i_s}.
$$
We have $\ord_\pp\LL_\nn=\ord_\pp x_{i_1}$. Also, we have
$$
\ord_\pp\NN\leq s\cdot\ord_\pp x_{i_1}<\ell\cdot\ord_\pp x_{i_1}=\ell\cdot\ord_\pp\LL_\nn,
$$
hence
$$
d=\deg(\DD)=\deg(\NN)\leq\ell\deg(\LL_\nn).
$$
\end{proof}
%%%%%%%%%%%%%%%%%%%%%%%%%%%%%%%%%%

%%%%%%%%%%%%%%%%%%%%%%%%%%%%%%%%%%
\begin{lemma}
We have
\begin{equation}\label{EqDelta}
(rx'_nx_n^{r-1}+a')^r=r^r\nu'^r(x_n^r+a)^{r-1}.
\end{equation}
\end{lemma}
%%%%%%
\begin{proof}
This can easily be derived from Pasten \cite[Equation (3.2)]{Pasten} through the obvious change of variables.
\end{proof}
%%%%%%%%%%%%%%%%%%%%%%%%%%%%%%%%%%

%%%%%%%%%%%%%%%%%%%%%%%%%%%%%%%%%%
\begin{notation}
Set $\Delta= a'^r-r^r\nu'^ra^{r-1}$ (this is just the ``part'' of Equation \eqref{EqDelta} that does not depend on $n$).
\end{notation}
%%%%%%%%%%%%%%%%%%%%%%%%%%%%%%%%%%

%%%%%%%%%%%%%%%%%%%%%%%%%%%%%%%%%%
\begin{lemma}\label{Wilson}%
If $\Delta\ne 0$ then the following inequality holds\,:
$$
\deg \dd(\Delta)\leq\frac{C_3r^2d}{M-1}\left(1+\frac{1}{r-1}\right).
$$
\end{lemma}%
%%%%%
\begin{proof}%
From Proposition \ref{yheight}, Lemma \ref{nuheight}, and Corollary \ref{Katia} it follows that
$$
\begin{aligned}
\deg \dd(\Delta)&\leq \max(rC_3\deg\dd(a),rC_3\deg\dd(\nu)+(r-1)\deg\dd(a))\\
&\leq rC_3(\deg\dd(a)+\deg\dd(\nu))\\
&\leq rC_3\left(\frac{rd}{M-1}+\frac{rd}{(r-1)(M-1)}\right )\\
&=\frac{C_3r^2d}{M-1}\left(1+\frac{1}{r-1}\right).
\end{aligned}
$$
\end{proof}%
%%%%%%%%%%%%%%%%%%%%%%%%%%%%%%%%%%

%%%%%%%%%%%%%%%%%%%%%%%%%%%%%%%%%%
\begin{notation}
\begin{enumerate}
\item Let us write $z=y^{2r+1}$ and $z_n=x_nz$.
\item Write
\[
C_5=(2r +1)C_4+\left(1+\frac{1}{r-1}\right)C_3=(2r+1)(4g+2) + \left(1+\frac{1}{r-1}\right)(3g+2)=
\]
\[
 \left (8r +4 +\frac{3r}{r-1}\right ) g + \left (4r +2 + \frac{2r}{r-1}\right )
\]
%$$
%=(8r+4+\frac{3r}{r-1})g +(4r+2+\frac{2r}{r-1})=\frac{8r^2-r-4}{r-1}g + \frac{8r^2-2}{r-1}.
%$$
and
$$
B=C_5r^2\ell+1=\beta_0(r,\ell)g + \beta_1(r, \ell),
$$
where
\[
\beta_0(r,\ell) =\left (8r +4 +\frac{3r}{r-1}\right )r^2\ell,
\]
\[
 \beta_1(r,\ell)=  \left (4r +2 + \frac{2r}{r-1}\right )r^2\ell + 1.
\]
\end{enumerate}
\end{notation}
%%%%%%%%%%%%%%%%%%%%%%%%%%%%%%%%%%

Observe that $\nn(z_n)$ is divisible by $\nn(x_n)$ because $z$ has a pole at $\pp_\infty$ only, and by assumption $\pp_\infty$ is not a zero of any $x_n$.

%%%%%%%%%%%%%%%%%%%%%%%%%%%%%%%%%%
\begin{lemma}\label{Chili}%
If $M>B(r,\ell)$ then $\Delta=0$, namely,
\begin{equation}\label{Joshua}
a'^r=r^r\nu'^ra^{r-1}.
\end{equation}
\end{lemma}
%%%%%%
\begin{proof}
Multiplying both sides of Equation \eqref{EqDelta}
$$
(rx'_nx_n^{r-1}+a')^r=r^r\nu'^r(x_n^r+a)^{r-1}
$$
by $z^{r^2}$ and replacing $x_n$ by $z_n=x_nz$ we get
\begin{equation}\label{Greenville}
(rx'_nz_n^{r-1}z+a'z^r)^r=r^r(\nu'z)^r(z_n^r+az^r)^{r-1}.
\end{equation}
Let $\pp$ be a prime of $K$ dividing $\nn(x_n)$. Let us remind the reader that by Corollary \ref{Nicolas} and definition of $z=y^{2r+1}$, the divisors $\dd(x_n)$, $\dd(a)$, $\dd(\nu)$, $\dd(x_n')$, $\dd(a')$ and $\dd(\nu')$ divide $\nn(z)$. Therefore, none of the terms $x_n'z$, $a'z^r$, $\nu'z$ and $az^r$ appearing in Equation \eqref{Greenville} have a pole at $\pp$.

We claim that $z^{r^2}\Delta$, that is, the part of Equation \eqref{Greenville} that does not depend on $n$, is divisible by $\nn(x_n)$. To see that, recall that $\nn(z_n)$ is divisible by $\nn(x_n)$, and hence $\nn(rx'_nz_n^{r-1}z)$ is divisible by $\nn(x_n)$ (see the left hand side of Equation \eqref{Greenville}). Also, there is no problem with the right hand side since the only part depending on $n$ is $z_n^r$. Thus modulo $\nn(x_n)$, Equation \eqref{Greenville} becomes
$$
(a'z^r)^r\equiv r^r(\nu'z)^r(az^r)^{r-1}\quad\mod\nn(x_n).
$$
Hence we have
$$
z^{r^2}\Delta=z^{r^2}(a'^r-r^r\nu'^r(az)^{r-1})\equiv0\quad\mod\nn(x_n)
$$

Thus, if $\Delta\ne0$ then from Corollary \ref{Francois} we have
$$
\begin{aligned}
\frac{d}{\ell}&\leq\deg\LL_\nn\\
&\leq\deg\dd(z^{r^2}\Delta)\\
&\leq r^2\deg\dd (y^{2r+1})+\deg\dd(\Delta)\\
&\leq r^2(2r+1) C_4\frac{d}{M-1}+\frac{C_3r^2d}{M-1}\left(1+\frac{1}{r-1}\right)\\
&=\frac{r^2d}{M-1}\left((2r +1)C_4+\left(1+\frac{1}{r-1}\right)C_3\right)\\
&=C_5\frac{r^2d}{M-1}.\\
\end{aligned}
$$
See Notation \ref{Groth} \eqref{Charlotte}, Proposition \ref{yheight} \eqref{deglcm} and Lemma \ref{Wilson}). Solving for $M$ we obtain
$$
M\leq C_5r^2\ell+1=B(r,\ell).
$$

So, for $M>B(r,\ell)$, the quantity $\Delta$ must be zero.
\end{proof}
%%%%%%%%%%%%%%%%%%%%%%%%%%%%%%%%%%

%%%%%%%%%%%%%%%%%%%%%%%%%%%%%%%%%%
\begin{remark}\label{NonZeroDer}
Note that from $\Delta=a'^r-r^r\nu'^ra^{r-1}=0$ we deduce that
$$
\nu'\ne0.
$$
Otherwise, both $\nu$ and $a$ would have zero derivative, which would imply by Equation \eqref{Gauss} that all $x_n$ have zero derivative and contradict the hypothesis of Theorem \ref{Savka}.
\end{remark}
%%%%%%%%%%%%%%%%%%%%%%%%%%%%%%%%%%

%%%%%%%%%%%%%%%%%%%%%%%%%%%%%%%%%%
\begin{proof}[Proof of Theorem \ref{Savka}]
Suppose $a$ is not zero. From Equation \eqref{Joshua}, the quantity
$$
a^{r-1}=\frac{a^r}{a}
$$
is an $r$-th power. Hence $a$ is an $r$-th power, say $a=b^r$.

On the one hand, from Equation \eqref{Joshua}, we have
$$
a'^r=r^r\nu'^rb^{r(r-1)}.
$$
Hence, taking an $r$-th root, we obtain
$$
a'=r\xi_0\nu'b^{r-1},
$$
where $\xi_0$ is an $r$-th root of unity.

On the other hand, from $a=b^r$, we have $a'=rb'b^{r-1}$, hence $\xi_0\nu'=b'$. Thus we get $b=\xi_0\nu+\gamma$ for some $\gamma\in K$ whose derivative is zero.

Finally from the Equation \eqref{Gauss}, we obtain
$$
x_n^r=(\nu+c_n)^r-a=(\nu+c_n)^r-b^r=(\nu+c_n)^r-(\xi_0\nu+\gamma)^r.
$$
\end{proof}

%%%%%%%%%%%%%%%%%%%%%%%%%%%%%%%%%%
%%%%%%%%%%%%%%%%%%%%%%%%%%%%%%%%%%
%%%%%%%%%%%%%%%%%%%%%%%%%%%%%%%%%%
%%%%%%%%%%%%%%%%%%%%%%%%%%%%%%%%%%
%%%%%%%%%%%%%%%%%%%%%%%%%%%%%%%%%%
%%%%%%%%%%%%%%%%%%%%%%%%%%%%%%%%%%
%%%%%%%%%%%%%%%%%%%%%%%%%%%%%%%%%%
%%%%%%%%%%%%%%%%%%%%%%%%%%%%%%%%%%
%%%%%%%%%%%%%%%%%%%%%%%%%%%%%%%%%%
%%%%%%%%%%%%%%%%%%%%%%%%%%%%%%%%%%

\section{Proof of Theorem \ref{main}}

From Theorem \ref{Savka}, we have
$$
x_n^r=(\nu+c_n)^r-a=(\nu+c_n)^r-(\xi_0\nu+\gamma)^r
$$
which is  polynomial in $\nu$ (and $\nu$ is non-constant by Remark \ref{NonZeroDer}), with coefficients in $F(\xi_0)$ (since $\gamma$ has zero derivative, it belongs to $F$). Therefore,
$$
a=(\nu+c_n)^r-x_n^r
$$
also is a polynomial in $\nu$ with coefficients in $F(\xi_0)$ and the problem is reduced to to  Hensley's Problem for polynomials in characteristic zero over $F(\xi_0)$. But we know that this problem has only trivial solutions for our $M$ (see Pasten \cite{Pasten}), implying that $a=0$. Contradiction.
%%%%%%%%%%%%%%%%%%%%%%%%%%%%%%%%%%

%%%%%%%%%%%%%%%%%%%%%%%%%%%%%%%%%%
%%%%%%%%%%%%%%%%%%%%%%%%%%%%%%%%%%
%%%%%%%%%%%%%%%%%%%%%%%%%%%%%%%%%%
%%%%%%%%%%%%%%%%%%%%%%%%%%%%%%%%%%
%%%%%%%%%%%%%%%%%%%%%%%%%%%%%%%%%%
%%%%%%%%%%%%%%%%%%%%%%%%%%%%%%%%%%
%%%%%%%%%%%%%%%%%%%%%%%%%%%%%%%%%%
%%%%%%%%%%%%%%%%%%%%%%%%%%%%%%%%%%
%%%%%%%%%%%%%%%%%%%%%%%%%%%%%%%%%%
%%%%%%%%%%%%%%%%%%%%%%%%%%%%%%%%%%

\section{Proof of Theorem \ref{main2}}

What remains to do in order to prove Theorem \ref{main2} is taken from \cite{PheidasVidaux2bis}, with essentially no changes. We include it here for the convenience of the reader.

In this section we let $r=2$ and $c_n=n$ for all $n$. Note that in this case $\ell(\bar c)=3$. For convenience of the reader, we rewrite Theorem \ref{Savka} under these assumptions\,:

%%%%%%%%%%%%%%%%%%%%%%%%%%%%%%%%%%
\begin{theorem}\label{Savka2}
Let $a,\nu\in K$, where $K$ is a function field of characteristic $p\geq B(2,3)$.  Let $M$ be a positive integer,  and let $(x_0,\dots, x_{M-1})$ be a sequence of elements of $K$ such that at least one $x_i$ is not a $p$-th power. If $M\geq B(2,3)$ and
\begin{equation}\label{Gauss2}
x_n^2=(\nu+n)^2-a,\, n=0,\dots,M-1,
\end{equation}
then either $a=0$ or $a=(\nu-\gamma)^2$ for some $\gamma\in K^p$.
\end{theorem}
%%%%%%%%%%%%%%%%%%%%%%%%%%%%%%%%%%

The rest of the section contains a proof of Theorem \ref{main2}. First we will dispose of the case where not all the $x_n$ are $p$-th powers. In this case Theorem \ref{Savka2} applies, namely there exists a $p$-th power $\gamma\in K$ such that
$$
x_n^2=(\nu+n)^2-(\nu-\gamma)^2.
$$
Write $\gamma=f^{p^s}$ so that $f\in K\setminus K^p$. For all $n$ we have
\begin{equation}\label{Boris}
\begin{aligned}
x_n^2&=(\nu+n)^2-(\nu-f^{p^s})^2\\
&=(2\nu-f^{p^s}+n)(f^{p^s}+n)\\
&=(2\nu-f^{p^s}+n)(f+n)^{p^s}\\
&=(2\nu-f^{p^s}+n)(f+n)(f+n)^{p^s-1}\\
&=\left[\left(\nu+\frac{f-f^{p^s}}{2}+n\right)^2-\left(\nu+\frac{f-f^{p^s}}{2}-f\right)^2\right]
(f+n)^{p^s-1}\\
\end{aligned}
\end{equation}
(note that for the third equality to hold, we need $c_n$ to be $n$). Considering the sequence defined by
$$
y_n=\frac{x_n}{(f+n)^{\frac{p^s-1}{2}}}
$$
we obtain
\begin{equation}\label{star}
y_n^2=\left(\bar\nu+n\right)^2-\left(\bar\nu-f\right)^2
\end{equation}
where
\begin{equation}\label{dag}
\bar\nu=\nu+\frac{f-f^{p^s}}{2}.
\end{equation}

We want to apply Theorem \ref{Savka2} to the sequence $y_n$. In order to do so, we show that $y_n$ cannot be a $p$-th power for more than one index $n$. Suppose that $y_n$ and $y_m$ are $p$-th powers for some distinct indices $n$ and $m$. Since
$$
y_n^2-y_m^2=(\bar\nu+n)^2-(\bar\nu-m)^2=2(n+m)\bar\nu+n^2-m^2,
$$
$\bar\nu$ is a $p$-th power. From Equation \eqref{star} we deduce that $(\bar\nu-f)^2$ is a $p$-th power, hence $\bar\nu-f$ is a $p$-th power, hence $f$ is a $p$-th power, and we have a contradiction of our assumption on $f$.

Since not all $y_n$ are $p$-th powers we may apply Theorem \ref{Savka2} to the sequence $(y_n)$. We assume that $\bar\nu-f\ne0$ and obtain a contradiction. Since $\bar\nu-f\ne0$, there exists a $p$-th power $\tilde\gamma$ such that $(\bar\nu-f)^2=(\bar\nu-\tilde\gamma)^2$. Since $f$ is not a $p$-th power, we have $f\ne\tilde\gamma$, hence
$$
\bar\nu-f=-\bar\nu+\tilde\gamma
$$
therefore,
$$
2\bar\nu=f+\tilde\gamma.
$$
From Equation \eqref{dag} we deduce
$$
f+\tilde\gamma=2\nu+f-f^{p^s}
$$
hence
$$
\tilde\gamma=2\nu-f^{p^s}.
$$
It follows that $\nu$ is a $p$-th power. Therefore, by Equations \eqref{Boris} we have
$$
\begin{aligned}
x_n^2&=(2\nu-f^{p^s}+n)(f^{p^s}+n)\\
&=(\tilde\gamma+n)(f^{p^s}+n)
\end{aligned}
$$
is a $p$-th power, hence also each $x_n$ is a $p$-th power. Thus we have a contradiction, implying $\bar\nu-f=0$.

From Equation \eqref{dag} we get
$$
f=\nu+\frac{f-f^{p^s}}{2}
$$
hence
$$
\nu=\frac{f+f^{p^s}}{2}
$$
and
$$
\begin{aligned}
x_n^2&=(2\nu-f^{p^s}+n)(f^{p^s}+n)\\
&=(f+n)(f^{p^s}+n)\\
&=(f+n)^{p^s+1}.
\end{aligned}
$$

Now we will address the case where all the $x_n$ are $p$-th powers. Under this assumption we consider the sequence $(w_n)$ such that for each $n$ we have $x_n=w_n^{p^h}$ and not all $w_n$ are $p$-th powers. So we may apply the above argument to the sequence $(w_n)$ (and the new corresponding values of $\nu$ and $\gamma$ - see \cite{PheidasVidaux2bis} for the details) and deduce that either $(w_n)$ is such that $w_n^2=(w+n)^2$ for some $w\in K$, or there exists $f\in K$ and a non-negative integer $s$ such that $w_n^2=(f+n)^{p^s+1}$. Therefore, either $x_n^2=(w^{p^h}+n)^2$, or
$$
x_n=\left[(f+n)^\frac{p^s+1}{2}\right]^{p^h}=(f^{p^h}+n)^\frac{p^s+1}{2}.
$$

It remains to verify that if the sequence $(x_n)$ satisfies Equations \eqref{Pasten} then it indeed satisfies Equations \eqref{Gausss}. Suppose that for each $n$ we have
$$
x_n=(f+n)^\frac{p^s+1}{2}
$$
for some $f\in K$ and $s$ a non-negative integer. Then we have
$$
\begin{aligned}
x_n^2&=(f+n)^{p^s+1}\\
&=(f+n)^{p^s}(f+n)\\
&=(f^{p^s}+n)(f+n)\\
&=\left(\frac{f^{p^s}+f}{2}+n\right)^2-\left(\frac{f^{p^s}-f}{2}\right)^2.
\end{aligned}
$$
which has the form $(x+n)^2+a$ for some polynomials $x$ and $a$ not depending on $n$.

%%%%%%%%%%%%%%%%%%%%%%%%%%%%%%%%%%
%%%%%%%%%%%%%%%%%%%%%%%%%%%%%%%%%%
%%%%%%%%%%%%%%%%%%%%%%%%%%%%%%%%%%
%%%%%%%%%%%%%%%%%%%%%%%%%%%%%%%%%%
%%%%%%%%%%%%%%%%%%%%%%%%%%%%%%%%%%
%%%%%%%%%%%%%%%%%%%%%%%%%%%%%%%%%%
%%%%%%%%%%%%%%%%%%%%%%%%%%%%%%%%%%
%%%%%%%%%%%%%%%%%%%%%%%%%%%%%%%%%%
%%%%%%%%%%%%%%%%%%%%%%%%%%%%%%%%%%
%%%%%%%%%%%%%%%%%%%%%%%%%%%%%%%%%%

\section{Proof of Corollary \ref{cor}}

The proof is similar to the proof of Theorem 1.8 in \cite{PheidasVidaux2} (this part of the proof was not affected by the mistake fixed in \cite{PheidasVidaux2bis}). We reproduce it here for the convenience of the reader.

Observe that in order to define multiplication, it is enough to define squaring. The following Lemmas \ref{Fritz3} and \ref{Fritz5} prove Corollary \ref{cor}.

Let $M\geq B(2,3)$ be an integer. Let $\phi(z,w)$ denote the formula
$$
\exists w_0,\dots,w_{M-1}
$$
$$
 \left[\bigwedge_{i=2,\dots,M-1}w_{i}-2w_{i-1}+w_{i-2}=2\bigwedge_{i=0,\dots,M-1}P_2(w_i)\wedge w=w_0\wedge2z=w_1-w_0-1\right]
$$
in the language $\calL_2$ (and thus also in the language $\calL^2_\tau$).  (We remind the reader that $P_2(w)$ denotes the predicate ``$w$ is a square''.)

It is clear that if $z,w\in K$ satisfy $z^2=w$, then $\phi(z,w)$ is true over $K$, since we can set $w_i=(z+i)^2$ for each $i=0,\dots,M-1$.  Observe that
$$
w_1-w_0-1=(z+1)^2-z^2-1=2z,
$$
and under our assumptions $(w_0,\ldots,w_{M-1})$ is a trivial B\"{u}chi sequence.

%%%%%%%%%%%%%%%%%%%%%%%%%%%%%%%%%%
\begin{lemma}\label{Fritz}
If $\phi(z,w)$ is satisfied over  $K$ for some $z$ and $w$ such that $z^2\not =w$ and $K$ has characteristic $0$, then $z$ and $w$ are in $F$. If $\phi(z,w)$ is satisfied over  $K$ for some $z$ and $w$ such that   $z^2\not =w$ and $K$ has characteristic $p\geq B(2,3)$, then either $z$ and $w$ are constant, or there exist $f\in K$ and a non-negative integer $s$ such that $w=f^{p^s+1}$ and $2z=f^{p^s}+f$.
\end{lemma}
%%%%%%
\begin{proof}
Suppose that $\phi(z,w)$ is true in $K$. Write $x_i^2=w_i$, so that we have $x_i^2-2x_{i-1}^2+x_{i-2}^2=2$ for each $i=2,\dots,M-1$. Writing $2\nu=\frac{x_n^2-x_0^2}{n}-n$ and $a=\nu^2-x_0^2$ we have $x_n^2=(\nu+n)^2-a$ for each $n$ (see Remark \ref{note}).

If $K$ has characteristic $0$, then by Theorem \ref{main} either $a=0$, and $\nu=\pm x_0$, so that
$$
2z=w_1-w_0-1=x_1^2-x_0^2-1=2\nu=\pm2x_0
$$
and $z^2=x_0^2=w_0=w$ contradicting our assumption, or for all indices $n$ we have that $w_n=x_n^2=(\nu+n)^2-a$ is in $F$, in which case $w=w_0\in F$ and $2z=w_1-w_0-1\in F$. Hence the first assertion of the Lemma is proved.

If $K$ has characteristic $p\geq B(2,3)$, then, as above,  by Theorem \ref{main2},  either $a=0$ and $z^2=w$ again contradicting our assumption, or for each index $n$ it is the case  that $(\nu+n)^2-a$ is in $F$ and thus  $w,z\in F$, or there exist $f\in K$ and a non-negative integer $s$ such that for each $n$, we have  $w_n=x_n^2=(f+n)^{p^s+1}$,  $w=w_0=f^{p^s+1}$ and
$$
2z=w_1-w_0-1=(f+1)^{p^s+1}-f^{p^s+1}-1=(f^{p^s}+1)(f+1)-f^{p^s+1}-1=f^{p^s}+f.
$$
\end{proof}
%%%%%%%%%%%%%%%%%%%%%%%%%%%%%%%%%%

%%%%%%%%%%%%%%%%%%%%%%%%%%%%%%%%%%
\begin{lemma}\label{Fritz3}
If $K$ has characteristic $0$, then it satisfies the formula of the language $\calL^2_\tau$
$$
\psi(z,w)\colon \phi(z,w)\wedge\phi(tz,t^2w)
$$
if and only if $z^2=w$ (where $tz$ stands for $\tau(z)$ and $t^2w$ stands for $\tau\tau w$).
\end{lemma}
%%%%%%
\begin{proof}
First we note that if $z,w\in K$ satisfy $z^2=w$, then the formula $\psi(z,w)$ is true in $K$ as was shown above. Suppose now that the formula $\psi(z,w)$ is satisfied in $K$ and that $z^2\ne w$ (hence $z,w\in F$). Since $\phi(tz,t^2w)$ is true in $K$, by Lemma \ref{Fritz} we have that either $(tz)^2=t^2w$ (which would contradict the hypothesis $z\ne w^2$), or both $tz$ and $t^2w$ are in $F$. Since $t$ stands for a transcendental element, this implies $z=w=0$, and in particular $z^2=w$. Contradiction.
\end{proof}
%%%%%%%%%%%%%%%%%%%%%%%%%%%%%%%%%%

%%%%%%%%%%%%%%%%%%%%%%%%%%%%%%%%%%
\begin{lemma}\label{Fritz4}
Suppose that $K$ has characteristic $p\geq B(2,3)$. If it satisfies the formula of the language $\calL^2_\tau$
$$
\theta(z,w)\colon\phi(z,w)\wedge\phi(z+t,w+2tz+z^2)\wedge\phi(z-t,w-2tz+z^2)
$$
and $z^2\ne w$ then either both $z$ and $w$ are $p$-th powers, or both $z+t$ and $w+2tz+z^2$ are $p$-th powers, or both $z-t$ and $w-2tz-z^2$ are $p$-th powers.
\end{lemma}
%%%%%%
\begin{proof}
See \cite[Section 3, Claim p. 563]{PheidasVidaux2}. Note that the proof is exactly the same since the expressions we have for $w$ and $z$ in Lemma \ref{Fritz} (2) are just special cases of the one used in \cite{PheidasVidaux2}.
\end{proof}
%%%%%%%%%%%%%%%%%%%%%%%%%%%%%%%%%%

%%%%%%%%%%%%%%%%%%%%%%%%%%%%%%%%%%
\begin{lemma}\label{Fritz5}
If $K$ has characteristic $p\geq B(2,3)$ then it satisfies the formula of the language $\calL_2^t$
$$
\eta(z,w)\colon\theta(z,w)\wedge\theta(z+t^2,w+2t^2z+t^4)
$$
if and only if $z^2=w$.
\end{lemma}
%%%%%%
\begin{proof}
It is a direct consequence of Lemma \ref{Fritz4} (or see \cite[Section 3, p. 563]{PheidasVidaux2}).
\end{proof}
%%%%%%%%%%%%%%%%%%%%%%%%%%%%%%%%%%

%%%%%%%%%%%%%%%%%%%%%%%%%%%%%%%%%%
%%%%%%%%%%%%%%%%%%%%%%%%%%%%%%%%%%
%%%%%%%%%%%%%%%%%%%%%%%%%%%%%%%%%%
%%%%%%%%%%%%%%%%%%%%%%%%%%%%%%%%%%
%%%%%%%%%%%%%%%%%%%%%%%%%%%%%%%%%%
%%%%%%%%%%%%%%%%%%%%%%%%%%%%%%%%%%
%%%%%%%%%%%%%%%%%%%%%%%%%%%%%%%%%%
%%%%%%%%%%%%%%%%%%%%%%%%%%%%%%%%%%

\end{document}